\documentclass [twoside,reqno, 12pt] {amsart}

\usepackage{amsfonts}
\usepackage{amssymb}
\usepackage{a4}
\usepackage{color}

\newtheorem{thm}{Theorem}[section]
\newtheorem{cor}[thm]{Corollary}
\newtheorem{lem}[thm]{Lemma}
\newtheorem{prop}[thm]{Proposition}

\newtheorem{rem}[thm]{Remark}

\theoremstyle{definition}

\numberwithin{equation}{section}

\newcommand{\C}{\mathbb{C}}

\newcommand{\N}{\mathbb{N}}

\newcommand{\R}{\mathbb{R}}

\newcommand{\supp}{\operatorname{supp}}

\newcommand{\Z}{\mathbb{Z}}

\def\hat{\widehat}
\def\tilde{\widetilde}
\def \bfo {\begin {eqnarray*} }
\def \efo {\end {eqnarray*} }
\def \ba {\begin {eqnarray*} }
\def \ea {\end {eqnarray*} }
\def \beq {\begin {eqnarray}}
\def \eeq {\end {eqnarray}}
\def \supp {\hbox{supp }}

\def \p {\partial}

\def\hat{\widehat}
\def\tilde{\widetilde}
\def \bfo {\begin {eqnarray*} }
\def \efo {\end {eqnarray*} }
\def \ba {\begin {eqnarray*} }
\def \ea {\end {eqnarray*} }
\def \beq {\begin {eqnarray}}
\def \eeq {\end {eqnarray}}
\def \supp {\hbox{supp }}

\def \p {\partial}

\begin{document}

 \title[stability at high frequencies]{Stability estimates in a partial data inverse boundary value problem for biharmonic operators at high frequencies}

\author[Liu]{Boya Liu}
\address{B. Liu, Department of Mathematics\\
University of California, Irvine\\ 
CA 92697-3875, USA }
\email{boyal2@uci.edu}

\maketitle

\begin{abstract}
We study the inverse boundary value problems of determining a potential in the Helmholtz type equation for the perturbed biharmonic operator from the knowledge of the partial Cauchy data set. Our geometric setting is that of a domain whose inaccessible portion of the boundary is contained in a hyperplane, and we are given the Cauchy data set on the complement. The uniqueness and logarithmic stability for this problem were established in \cite{Yang-Yang} and \cite{Ch_Heck_2017}, respectively. We establish stability estimates in the high frequency regime,  with an explicit dependence on the frequency parameter, under mild regularity assumptions on the potentials, sharpening those of \cite{Ch_Heck_2017}. 
\end{abstract}

\section{Introduction and statement of results}

Let $\Omega\subset \{x = (x_1, x_2, \dots, x_n) \in \R^n: x_n<0\}$, $n\ge 3$, be a bounded open set with $C^\infty$ boundary. Assume that $\Gamma_0:=\p \Omega\cap \{x_n=0\}$ is non-empty, and let us set $\Gamma=\p \Omega\setminus \Gamma_0$. Let $k\ge 0$, and let $q\in L^\infty(\Omega)$.  Consider the Helmholtz type equation for the perturbed biharmonic operator, 
\begin{equation}
\label{eq_int_1}
(\Delta^2-k^4 +q)u =0\quad \text{in}\quad \Omega. 
\end{equation}
Such fourth order operators arise in the context of modeling of hinged elastic beams and suspension bridges, see \cite{Gaz_Grun_Sweers_book}. Associated to the equation \eqref{eq_int_1} and  the open portion $\Gamma$ of $\p \Omega$, we introduce the partial Cauchy data set
\begin{align*}
&C_q^\Gamma(k)=\{ (u|_{\Gamma}, (\Delta u)|_{\Gamma}, \p_\nu u|_{\Gamma}, \p_\nu(\Delta u)|_{\Gamma}) : u\in H^4(\Omega)\text{ satisfies }\eqref{eq_int_1}, \\
&u|_{\Gamma_0}=(\Delta u)|_{\Gamma_0}=0\}\subset H^{\frac{7}{2}}(\Gamma)\times H^{ \frac{3}{2}}(\Gamma)\times H^{\frac{5}{2}}(\Gamma)\times H^{\frac{1}{2}}(\Gamma)=:H^{\frac{7}{2}, \frac{3}{2},\frac{5}{2},\frac{1}{2}}(\Gamma).
\end{align*} 
Here  and in what follows, $\nu$ is the unit outer normal to $\p \Omega$, and  $H^s(\Omega)=\{U|_{\Omega}:U\in H^s(\R^n)\}$, $s\in \R$, is the standard Sobolev space on $\Omega$, see \cite[Chapter 5]{Agranovich_book}. 

In this paper we are concerned with the partial data inverse problem of recovering the potential $q$ in $\Omega$ from the knowledge of the partial Cauchy data set $C_q^\Gamma(k)$ at a fixed frequency $k\ge 0$. The global uniqueness for this problem was established in \cite{Yang-Yang}.  Stability estimates of logarithmic type complementing the uniqueness result were obtained in \cite{Ch_Heck_2017}, when $k=0$, under the assumption that  $q\in H^s(\Omega)$, 
for some $s>\frac{n}{2}$.  Thanks to the work \cite{Mandache_2001}, logarithmic stability estimates are expected to be optimal for such inverse boundary value problems  when the operator $\Delta^2$ is replaced by $\Delta$ and $k=0$, even when the full Cauchy data set $C_q^{\p \Omega}(k)$ is given. We refer to  the works  \cite{Assylbekov_Iyer}, \cite{Assylbekov_Yang}, \cite{Bhatt_Ghosh}, \cite{Ch_Venki_2015},   \cite{Ghosh_2016}, \cite{Ikehata_1991}, \cite{Isakov_1991}, \cite{KLU_2012},   \cite{KLU_2014}, \cite{Krup_Uhlmann_2016_spectral}, \cite{Serov}, among others, for the study of inverse boundary value problems for perturbed biharmonic operators. 

The logarithmic stability estimates established in \cite{Ch_Heck_2017} indicate that the inverse boundary value problem in question is severely ill--posed, which makes it most challenging to design reconstruction algorithms with high resolution in practice, since small errors in measurements may result in exponentially large errors in the reconstruction of the unknown potential. Nevertheless, it has been observed numerically that the stability may increase when the frequency $k$ of the problem becomes large, see \cite{Colton_Haddar}. Stability estimates at high frequencies as well as the  phenomenon of increasing stability  in the high frequency regime  for several fundamental inverse boundary value problems, have been studied rigorously in \cite{Hry_Isakov}, \cite{Isakov_2007_inres} \cite{Isakov_2011}, \cite{Isakov_Lai_Wang_2016}, \cite{INUW_2014}, \cite{Isakov_Wang_2014}, among others, in the full data case. In the case of partial data inverse boundary value  problems, the question of deriving 
stability estimates at large frequencies and understanding the increasing stability phenomenon has only been studied in the works \cite{Ch_Heck_2018},  \cite{Krupchyk_Uhlmann_2019}, and \cite{Liang}, all in the case of Schr\"odinger operators, to the best of our knowledge.

The goal of this paper is to study the issue of deriving stability estimates in the high frequency regime for the partial data inverse problem formulated above, for the perturbed biharmonic operator. To state our results, let $q_1, q_2\in (L^\infty\cap H^s)(\Omega)$, for some $0<s<1/2$.  We define the distance between the partial Cauchy data sets as follows, 
\begin{align*}
\text{dist} (C_{q_1}^\Gamma(k), C_{q_2}^\Gamma(k)):=\max\bigg\{
\sup_{0\ne f\in C_{q_1}^\Gamma(k)}\inf_{\tilde f\in C_{q_2}^\Gamma(k)} \frac{\|f-\tilde f\|_{H^{\frac{7}{2}, \frac{3}{2},\frac{5}{2},\frac{1}{2}}(\Gamma)}}{\|f\|_{H^{\frac{7}{2}, \frac{3}{2},\frac{5}{2},\frac{1}{2}}(\Gamma)}}, \\
\sup_{0\ne f\in C_{q_2}^\Gamma(k)}\inf_{\tilde f\in C_{q_1}^\Gamma(k)} \frac{\|f-\tilde f\|_{H^{\frac{7}{2}, \frac{3}{2},\frac{5}{2},\frac{1}{2}}(\Gamma)}}{\|f\|_{H^{\frac{7}{2}, \frac{3}{2},\frac{5}{2},\frac{1}{2}}(\Gamma)}}
 \bigg\},
\end{align*}
where $f=(f_1, f_2, f_3, f_4)$ and $\tilde f=(\tilde f_1, \tilde f_2, \tilde f_3, \tilde f_4)$, and the norm in the space $H^{\frac{7}{2}, \frac{3}{2},\frac{5}{2},\frac{1}{2}}(\Gamma)$ is given by
\[
\|f\|_{H^{\frac{7}{2}, \frac{3}{2},\frac{5}{2},\frac{1}{2}}(\Gamma)}= \big(\|f_1\|_{H^{\frac{7}{2}}(\Gamma)}^2+ \|f_2\|_{H^{ \frac{3}{2}}(\Gamma)}^2+\|f_3\|_{H^{\frac{5}{2}}(\Gamma)}^2+\|f_4\|_{H^{\frac{1}{2}}(\Gamma)}^2\big)^{1/2}.
\]

The main result of this paper is as follows. 
\begin{thm}
\label{thm_main}
Let $\Omega\subset \{\R^n: x_n<0\}$, $n\ge 3$, be a bounded open set with $C^\infty$ boundary. Assume that $\Gamma_0=\p \Omega\cap \{x_n=0\}$ is non-empty, and let $\Gamma=\p \Omega\setminus \Gamma_0$. Let $M>0$, $0<s<1/2$,   and let $q_1, q_2\in (L^\infty\cap H^s)(\Omega)$  be such that $\|q_j\|_{L^\infty(\Omega)}+\|q_j\|_{H^s(\Omega)}\le M$, $j=1,2$. Then there is a constant $C>0$ such that for all $k\ge 1$, $0<\delta:= \emph{\text{dist}} (C_{q_1}^\Gamma(k), C_{q_2}^\Gamma(k))<1/e$, we have
\begin{equation}
\label{eq_int_2}
\|q_1-q_2\|_{H^{-1}(\Omega)}\le   e^{Ck}\delta^{\frac{1}{2}}+ \frac{C}{(k+\log \frac{1}{\delta})^{2\alpha}}. 
\end{equation}
Here $0<\alpha=\frac{(n-1)s}{(2s+n-1)(n+2)} <\frac{1}{2}$ and $C>0$ depends on $\Omega$, $M$, $s$,  but is independent of $k$. 
\end{thm}

\begin{rem} To the best of our knowledge, Theorem \ref{thm_main}  seems to be the first stability estimate at high frequencies in the context of inverse boundary value problems for higher order elliptic PDE. 
\end{rem}

\begin{rem} Theorem \ref{thm_main} implies also a stability result for a fixed frequency $k$, which is sharper than the 
stability result of \cite{Ch_Heck_2017}, in terms of the regularity of the potentials. Indeed, in \cite{Ch_Heck_2017} one assumes that the potentials are of class $H^s(\Omega)$, $s>\frac{n}{2}$, whereas in Theorem \ref{thm_main}, the $(L^\infty \cap H^s)(\Omega)$--regularity, $0<s<1/2$, suffices. In particular, no continuity of the potentials in Theorem \ref{thm_main} is assumed, and the required Sobolev regularity assumptions are fairly mild and are independent of the dimension. 
\end{rem}

\begin{rem} 
The partial data inverse problem in the particular setting considered in this paper, i.e. when the inaccessible portion of the boundary of the domain $\Omega$ is a part of the hyperplane and the measurements are performed on its complement, was first studied in \cite{Isakov_local} in the context of the Schr\"odinger equation. The uniqueness result of  \cite{Isakov_local} was complemented by the logarithmic stability estimates of \cite{Heck_Wang} for potentials of class $H^s(\Omega)$, $s>\frac{n}{2}$, see also \cite{Caro_Marinov_2016}. High frequency stability results were obtained in \cite{Ch_Heck_2018}, for potentials of the same class, and increasing stability estimates were obtained in \cite{Liang} for potentials of class $C^1(\Omega)$, assuming also the potentials agree near the flat portion of the boundary.  We would like to emphasize that the method of the proof of Theorem  \ref{thm_main} allows one to improve the regularity assumptions on the potentials to $(L^\infty\cap H^s)(\Omega)$, $0<s<1/2$, in all of the stability and  increasing stability results for the Schr\"odinger equation.  This is accomplished by deriving and exploiting a quantitative version of the Riemann--Lebesgue lemma, under mild regularity assumptions on the potentials, whereas the aforementioned works rely on a quantitative version of the Riemann--Lebesgue lemma valid for  $L^1$ potentials with the $L^1$--modulus of continuity of H\"older type, established in \cite{Heck_Wang}.
\end{rem}

The strategy of the proof of Theorem  \ref{thm_main}, going back to \cite{Isakov_local}, consists of using a reflection argument to construct complex geometric optics solutions to the biharmonic equation, which vanish on the flat portion of the boundary.  A fundamental role in \cite{Isakov_local} is also played by the Riemann-Lebesgue lemma, which should be sharpened to a quantitative statement when deriving stability estimates. Here we establish an accurate version of the Riemann-Lebesgue lemma for compactly supported elements of $H^s(\R^n)$, $0< s < 1/2$. The advantage of working with potentials of such Sobolev regularity, which we exploit, comes from the well known fact that the operator of extension by zero takes $H^s(\Omega)$ to $H^s(\R^n)$, $0< s <1/2$, boundedly,    see \cite[Theorem 5.1.10]{Agranovich_book} and \cite{Faraco_Rogers}.  It is thanks to this observation that we are able to establish Theorem  \ref{thm_main} under fairly mild regularity assumptions.

Assuming that the potentials $q_1$ and $q_2$ enjoy some additional regularity properties and a priori bounds, we obtain the following corollary of Theorem \ref{thm_main}. 
\begin{cor}
\label{cor_main}
Let $\Omega\subset \{\R^n: x_n<0\}$, $n\ge 3$, be a bounded open set with $C^\infty$ boundary. Assume that $\Gamma_0=\p \Omega\cap \{x_n=0\}$ is non-empty, and let $\Gamma=\p \Omega\setminus \Gamma_0$.  Let $M>0$, $s>\frac{n}{2}$,  and let $q_1, q_2\in H^s(\Omega)$, be such that $\|q_j\|_{H^s(\Omega)}\le M$, $j=1,2$. Then there is a constant $C>0$ such that for all $k\ge 1$, $0<\delta:= \emph{\text{dist}} (C_{q_1}^\Gamma(k), C_{q_2}^\Gamma(k))<1/e$, we have
\begin{equation}
\label{eq_int_3}
\|q_1-q_2\|_{L^\infty(\Omega)}\le  \bigg( e^{Ck}\delta^{\frac{1}{2}} + \frac{C}{(k+\log \frac{1}{\delta})^{2\alpha}}
\bigg)^{\frac{s-\frac{n}{2}}{2(s+1)}}.
\end{equation}
Here $0<\alpha=\frac{(n-1)s}{(2s+n-1)(n+2)} <\frac{1}{2}$ and $C>0$ depends on $\Omega$, $M$, $s$,  but is independent of $k$. 
\end{cor}

This paper is organized as follows. Section \ref{app_CGO} contains a construction of complex geometric optics solutions to Helmholtz type equations for perturbed biharmonic operators by means of a Fourier series approach, extending the result of  \cite{Hahner_1996} in the Schr\"odinger case. Let us remark that the original approach to the construction of such solutions in \cite{Sylvester_Uhlmann_1987} proceeds globally, by constructing the Faddeev Green function and exploiting techniques of the Fourier transformation in all of $\R^n$. See also \cite{Krup_Uhlmann_2016_spectral} for an extension of such constructions to the case of polyharmonic operators with (unbounded) potentials. In Section \ref{sec_Riemann_Lebesgue} we establish a quantitative version of the Riemann--Lebesgue Lemma for compactly supported functions in $H^s(\R^n)$, $0< s < 1/2$. The proofs of Theorem \ref{thm_main} and Corollary \ref{cor_main} are given in Section \ref{sec_proof}.

\section{Complex geometric optics solutions to Helmholtz type equations for perturbed biharmonic operators}

\label{app_CGO}

We have the following result in the spirit of  
\cite{Hahner_1996}, see also   \cite{FSU_book}, for the Helmholtz type equations for perturbed  biharmonic operators.  This result is useful here since all the constants are independent of the frequency $k$. 
\begin{prop}
\label{prop_cgo_Hahner}
Let $\Omega\subset \R^n$, $n\ge 2$, be a bounded open set, let $q\in L^\infty(\Omega)$, and $k\ge 0$. Then there are constants $C_0>0$ and $C_1>0$, depending on $\Omega$ and $n$ only, such that  for all $\zeta\in \C^n$,  $\zeta\cdot\zeta=k^2$, and $|\emph{\text{Im}} \zeta|\ge \max \{C_0\sqrt{\|q\|_{L^\infty(\Omega)}}, 1\}$, the equation
\begin{equation}
\label{CGO_1}
(\Delta^2-k^4+q)u=0\quad \text{in}\quad \Omega
\end{equation}
has a solution of the form
\begin{equation}
\label{CGO_2}
u(x;\zeta)=e^{i\zeta\cdot x}(1+r(x;\zeta)),
\end{equation}
where $r\in L^2(\Omega)$ satisfies 
\[
\|r(\cdot;\zeta)\|_{L^2(\Omega)}\le \frac{C_1}{|\emph{\text{Im}} \zeta|^2}\|q\|_{L^\infty(\Omega)}.
\]
\end{prop}

\begin{proof}
We shall follow \cite{Hahner_1996}, see also   \cite{FSU_book}. Let $\zeta\in \C^n$ be such that $\zeta\cdot\zeta=k^2$.  Then we have
\[
e^{-i\zeta\cdot x}\circ \Delta^2\circ e^{i\zeta\cdot x}=(D^2+2\zeta\cdot D)^2 +2k^2(D^2+2\zeta\cdot D)  +k^4,
\]
where $D_{x_j}=\frac{1}{i}\p_{x_j}$.
Thus, \eqref{CGO_2} is a solution to \eqref{CGO_1} if and only if 
\begin{equation}
\label{CGO_3}
[(D^2+2\zeta\cdot D)^2 +2k^2(D^2+2\zeta\cdot D) +q]r=-q\quad \text{in}\quad \Omega. 
\end{equation}
Writing $\zeta=w_1+iw_2$, where $w_1, w_2\in \R^n$, and using the fact that $\zeta\cdot\zeta=k^2$, we see that 
$w_1\cdot w_2=0$.  Performing an orthogonal transformation, we may therefore assume that  $w_1=|w_1|e_1$ and $w_2=|w_2|e_2$, where $e_1$ and $e_2$ are the first two vectors in the standard basis of $\R^n$.  In order to solve the equation \eqref{CGO_3}, let us first solve the following equation,
\begin{equation}
\label{CGO_4}
[(D^2+2\zeta\cdot D)^2 +2k^2(D^2+2\zeta\cdot D) ]r=f\quad \text{in}\quad \Omega. 
\end{equation}
where $f\in L^2(\Omega)$. 
In doing so let us assume for simplicity that $\Omega\subset Q:=[-\pi, \pi]^n$. In the following everything works without this extra assumption if we replace the set $\Omega$ by its image under the map $\R^n\ni x\mapsto\kappa x\in \R^n$ for some fixed $\kappa>0$ sufficiently small. Let us extend $f$ by zero outside of $\Omega$ into $Q$, and let us denote this extension again by $f$.  Thus, it suffices to solve the following equation, 
\begin{equation}
\label{CGO_5}
[(D^2+2|w_1|D_{x_1}+2i|w_2|D_{x_2})^2 +2k^2(D^2+2|w_1|D_{x_1}+2i |w_2|D_{x_2}) ]r=f\quad \text{in}\quad Q.
\end{equation}
To that end, let 
\[
v_l(x)=e^{i(l+\frac{1}{2}e_2)\cdot x}, \quad l\in \Z^n.
\]
The set $(v_l)_{l\in \Z^n}$ is an orthonormal basis in  $L^2(Q)$, see  \cite{FSU_book},  and therefore, $f$ can be written as a series 
\[
f=\sum_{l\in \Z^n} f_l v_l,
\]
where $f_l=(f,v_l)_{L^2(Q)}=(2\pi)^{-n}\int_{Q} f\overline{v_l}dx$, $\|f\|_{L^2(Q)}^2=\sum_{l\in \Z^n}|f_l|^2$. The reason for considering a shifted integer lattice rather than the standard integer coordinate lattice is due to the fact that the symbol of the operator in \eqref{CGO_5} is non-vanishing along the shifted lattice, see \eqref{CGO_5_new_100} below,   and the equation \eqref{CGO_5} can therefore be solved by division.

We look for a solution of \eqref{CGO_5} in the form
\[
r=\sum_{l\in \Z^n} r_l v_l,
\] 
and therefore, \eqref{CGO_5} leads to the following equation,
\[
(p_l^2+2k^2p_l)r_l=f_l, 
\]
where 
\[
p_l= \bigg(l+\frac{1}{2}e_2\bigg)^2+2|w_1|l_1 +2i |w_2|\bigg(l_2+\frac{1}{2}\bigg). 
\]
We have
\begin{equation}
\label{CGO_5_new_100}
|\text{Im}\, p_l|=|\text{Im}\, (p_l+2k^2)|=2 |w_2|\bigg|l_2+\frac{1}{2}\bigg|\ge |w_2|, \quad l_2\in \Z.
\end{equation}
Assume that $|w_2|\ne 0$. Letting  
\[
r_l:=\frac{f_l}{p_l^2+2k^2 p_l},
\]
we see that 
\begin{equation}
\label{CGO_6}
|r_l|\le \frac{|f_l|}{|w_2|^2},
\end{equation}
and therefore, $\|r\|_{L^2(Q)}\le \frac{1}{|w_2|^2}\|f\|_{L^2(Q)}$. Thus, we have shown that for any $\zeta\in \C^n$ such that $\zeta\cdot\zeta=k^2$ and $|\text{Im}\, \zeta|\ne 0$, the equation \eqref{CGO_4} has a solution operator, 
\begin{equation}
\label{CGO_6_1}
G_\zeta:  L^2(Q)\to L^2(Q), \quad f\mapsto r,
\end{equation}
such that $\|G_\zeta\|_{\mathcal{L}(L^2(Q),  L^2(Q))}\le \frac{1}{|\text{Im}\, \zeta|^2}$.  

Let us now return to the equation \eqref{CGO_3} and look for a solution in the form $r=G_\zeta \tilde r$, where $\tilde r\in L^2(Q)$ is to be determined. Then we get 
\[
(I+qG_\zeta)\tilde r=-q\quad \text{in}\quad L^2(Q). 
\]
Since $\|qG_\zeta\|_{\mathcal{L}(L^2(Q),  L^2(Q))}\le 1/2$ provided that $|\text{Im\,} \zeta|\ge \sqrt{2\|q\|_{L^\infty(Q)}}$, the operator $I+qG_\zeta$ is invertible on $L^2(Q)$ and $\tilde r=(I+qG_\zeta)^{-1}(-q)$. By the Neumann series, $\|\tilde r\|_{L^2(Q)}\le 2\|q\|_{L^2(Q)}$. Thus, $\|r\|_{L^2(Q)}\le \frac{2}{|\text{Im\,}\zeta|^2}\|q\|_{L^2(Q)}$.  This completes the proof of Proposition \ref{prop_cgo_Hahner}. 
\end{proof}

\section{Quantitative version of the Riemann--Lebesgue Lemma}
\label{sec_Riemann_Lebesgue}

The goal of this section is to prove a quantitative version of the Riemann--Lebesgue lemma for functions $f\in H^s(\R^n)$, $0< s <1$, with $\supp(f)$ compact. Another closely related version of the Riemann--Lebesgue lemma is established in \cite{Heck_Wang}, where it is applied to zero extensions of H\"older continuous functions defined on smooth bounded domains. 

In what follows let $\Psi_\tau(x)=\tau^{-n}\Psi (x/\tau)$, $\tau>0$, be the usual mollifier with $\Psi\in C^\infty_0(\R^n)$, $0\le \Psi\le 1$, and $\int_{\R^n} \Psi dx=1$. 

We shall need the following approximation result, which was established in \cite{Krupchyk_Uhlmann_2018}. 
\begin{lem}
\label{lem_approximation}
Let $f\in H^s(\R^n)$, $0\le s< 1$. Then $f_\tau=f*\Psi_\tau\in( C^\infty\cap H^s)(\R^n)$, and 
\[
\|f-f_\tau\|_{L^2(\R^n)}=o(\tau^s), \quad \tau\to 0. 
\]
\end{lem}

The following result is a quantitative version of the Riemann--Lebesgue lemma. 
\begin{prop}
\label{prop_Riemann_Lebesgue}
Let $f\in H^s(\R^n)$, $0<s<1$, be such that $\supp (f)$ is compact. Then there exists constant $C>0$ and for any $N\in \N$, there exists $C_N>0$ such that for all $\xi\in \R^n$ and $0<\tau<1$,  we have 
\begin{equation}
\label{eq_4_3_0}
|\hat f (\xi)|\le \frac{C_N}{(1+ \tau |\xi|)^{N}}+ C\tau^s. 
\end{equation}
\end{prop}
\begin{proof}

We have 
\[
|\hat f(\xi)|\le |\hat f_\tau(\xi)|+  |\hat f_\tau(\xi)-\hat f(\xi)|.
\]
Using that $\hat \Psi_\tau(\xi)=\hat \Psi(\tau \xi)$, we get $\hat f_\tau(\xi)=\hat f(\xi)\hat \Psi(\tau \xi)$. As $\supp (f)$ is compact, we see that $f\in L^1(\R^n)$, and therefore, 
\begin{equation}
\label{eq_4_4}
| \hat f_\tau(\xi)|\le \|f\|_{L^1(\R^n)}|\hat \Psi(\tau \xi)|\le C\|f\|_{H^s(\R^n)}|\hat \Psi(\tau \xi)|.
 \end{equation}
 As $\hat \Psi\in \mathcal{S}(\R^n)$, we get 
\begin{equation}
\label{eq_4_5}
|\hat \Psi(\tau \xi)|\le \frac{C_N}{( 1+ \tau|\xi|)^N} 
\end{equation}
for all $\xi\in \R^n$, $\tau >0$, and $N\in \N$. Combining \eqref{eq_4_4} and \eqref{eq_4_5}, we obtain that 
\[
 |\hat f_\tau(\xi)|\le\frac{C_N}{(1+ \tau|\xi|)^N}
\]
for all $\xi\in \R^n$, $\tau >0$, and $N\in \N$. 

Using Young's inequality, we see that $\|f_\tau - f\|_{L^2(\R^n)}\le 2\|f\|_{L^2(\R^n)}$ for  $\tau\in (0,\infty)$. Combining this with the fact that $\supp(f)$ is compact and using Lemma \ref{lem_approximation}, we get for all $0<\tau<1$, 
\[
|\hat f_\tau(\xi)-\hat f(\xi)|\le \|f_\tau - f\|_{L^1(\R^n)} \le C\|f_\tau - f\|_{L^2(\R^n)}\le C\tau ^s.  
\]
This completes the proof of Proposition \ref{prop_Riemann_Lebesgue}.
\end{proof}

\section{Proofs of Theorem \ref{thm_main} and Corollary \ref{cor_main}}

\label{sec_proof}

\subsection{Derivation of the integral inequality}

\begin{lem}
\label{lem_int_identity}
Let $q_1, q_2\in L^\infty(\Omega)$ and $k\ge 0$. We have 
\begin{align*}
\bigg|\int_{\Omega}(q_2&-q_1)u_1u_2 dx\bigg|
\le 4 \| (u_1|_{\Gamma}, (\Delta u_1)|_{\Gamma}, (\p_\nu u_1)|_{\Gamma}, \p_\nu(\Delta u_1)|_{\Gamma})  \|_{H^{\frac{7}{2}, \frac{3}{2},\frac{5}{2},\frac{1}{2}}(\Gamma)}\\
 &\| (u_2|_{\Gamma}, (\Delta u_2)|_{\Gamma}, (\p_\nu u_2)|_{\Gamma}, \p_\nu(\Delta u_2)|_{\Gamma})  \|_{H^{\frac{7}{2}, \frac{3}{2},\frac{5}{2},\frac{1}{2}}(\Gamma)} \emph{\text{dist}}(C_{q_1}^\Gamma(k), C_{q_2}^\Gamma (k)),
\end{align*}
for any $u_1, u_2\in H^4(\Omega)$ such that 
\begin{align*}
(\Delta^2-k^4+q_1)u_1=0 \quad \text{in}\quad \Omega, \quad u_1|_{\Gamma_0}=(\Delta u_1)|_{\Gamma_0}=0,\\
(\Delta^2-k^4+q_2)u_2=0 \quad \text{in}\quad \Omega, \quad u_2|_{\Gamma_0}=(\Delta u_2)|_{\Gamma_0}=0.
\end{align*}
\end{lem}

\begin{proof}
We shall need the following Green's formula, see \cite{Grubb_book},
\begin{equation}
\label{eq_2_1}
\begin{aligned}
\int_{\Omega} (\Delta^2 u)v dx-\int_{\Omega} u(\Delta^2v)dx&=\int_{\p \Omega}\p_\nu(\Delta u) vdS-\int_{\p \Omega} (\Delta u)\p_\nu v dS\\
& +\int_{\p \Omega} \p_\nu u (\Delta v)dS -\int_{\p \Omega} u\p_\nu (\Delta v)dS,
\end{aligned}
\end{equation}
 valid for  $u,v\in H^4(\Omega)$. 
 
 Let $u_1, u_2\in H^4(\Omega)$ be solutions to 
\begin{equation}
\label{eq_2_2}
(\Delta^2-k^4+q_1)u_1=0, \quad (\Delta^2-k^4+q_2)u_2=0, \quad \text{in}\quad \Omega, 
\end{equation}
respectively, such that $u_1|_{\Gamma_0}=(\Delta u_1)|_{\Gamma_0}=0$ and $u_2|_{\Gamma_0}=(\Delta u_2)|_{\Gamma_0}=0$. Multiplying the first equation in \eqref{eq_2_2} by $u_2$ and using the Green's formula \eqref{eq_2_1}, we obtain that 
\begin{equation}
\label{eq_2_3}
\begin{aligned}
\int_{\Omega}(q_2-q_1)u_1u_2 dx=& \int_{\Gamma}\p_\nu(\Delta u_1) u_2dS-\int_{\Gamma} (\Delta u_1)\p_\nu u_2 dS \\
&+\int_{\Gamma} \p_\nu u_1 (\Delta u_2)dS -\int_{\Gamma} u_1\p_\nu (\Delta u_2)dS.
\end{aligned}
\end{equation}
Let $f=(f_1,f_2, f_3, f_4)\in C_{q_1}^\Gamma(k)$ be arbitrary. Then there exists $v\in H^4(\Omega)$ such that 
\begin{equation}
\label{eq_2_4}
\begin{aligned}
&(\Delta^2-k^4+q_1)v=0\quad \Omega, \\
&v|_{\Gamma_0}=(\Delta v)|_{\Gamma_0}=0, \\
&v|_{\Gamma}=f_1,\quad  (\Delta v)|_{\Gamma}=f_2,\quad (\p_\nu v)|_{\Gamma}=f_3, \quad \p_\nu(\Delta v)|_{\Gamma}=f_4.  
\end{aligned}
\end{equation}

Indeed, this follows from the definition of the set of the Cauchy data $C_{q_1}^\Gamma(k)$ together with the fact that the boundary value problem,
\begin{equation}
\label{eq_2_4_new_100}
\begin{aligned}
&(\Delta^2-k^2+q_1)u=0\quad \text{in}\quad \Omega,\\
&u|_{\p \Omega}=g_1\in H^{\frac{7}{2}}(\p \Omega),\\
&(\Delta u)|_{\p \Omega}=g_2\in H^{\frac{3}{2}}(\p \Omega),
\end{aligned}
\end{equation} 
enjoys the Fredholm property. The latter may be seen, for instance, by rewriting   \eqref{eq_2_4_new_100} as a boundary value problem for a strongly elliptic system and applying \cite[Theorem 4.10]{McLean_book} together with elliptic boundary regularity.

Multiplying the first equation in \eqref{eq_2_2} by $v$,  using the Green formula \eqref{eq_2_1} and \eqref{eq_2_4}, we get
\begin{equation}
\label{eq_2_5}
\begin{aligned}
\int_{\Gamma}\p_\nu(\Delta u_1) f_1dS
 -\int_{\Gamma} (\Delta u_1)f_3 dS +\int_{\Gamma} \p_\nu u_1 f_2dS -\int_{\Gamma} u_1 f_4dS=0.
\end{aligned}
\end{equation}
Combining \eqref{eq_2_3} and \eqref{eq_2_5}, we get 
\begin{equation}
\label{eq_2_6}
\begin{aligned}
\int_{\Omega}(q_2-q_1)u_1u_2 dx=& \int_{\Gamma}\p_\nu(\Delta u_1) (u_2-f_1)dS-\int_{\Gamma} (\Delta u_1)(\p_\nu u_2 -f_3)dS \\
&+\int_{\Gamma} \p_\nu u_1 (\Delta u_2-f_2)dS -\int_{\Gamma} u_1(\p_\nu (\Delta u_2)-f_4)dS.
\end{aligned}
\end{equation}
Letting 
\[
\|f\|_{L^2(\Gamma)}=(\|f_1\|_{L^2}^2+\|f_2\|_{L^2}^2+\|f_3\|_{L^2}^2+ \|f_4\|_{L^2}^2)^{1/2},
\]
we see from \eqref{eq_2_6} that 
\begin{align*}
\bigg|\int_{\Omega}(q_2-q_1)u_1u_2 dx\bigg|\le 4 \| (u_1|_{\Gamma}, (\Delta u_1)|_{\Gamma}, (\p_\nu u_1)|_{\Gamma}, \p_\nu(\Delta u_1)|_{\Gamma})  \|_{L^2(\Gamma)}\\
\| (u_2|_{\Gamma}-f_1, (\Delta u_2)|_{\Gamma}-f_2, (\p_\nu u_2)|_{\Gamma}-f_3, \p_\nu(\Delta u_2)|_{\Gamma}-f_4)  \|_{L^2(\Gamma)},
\end{align*}
and therefore, 
\begin{equation}
\label{eq_2_7}
\begin{aligned}
\bigg|&\int_{\Omega}(q_2-q_1)u_1u_2 dx\bigg|\le  4 \| (u_1|_{\Gamma}, (\Delta u_1)|_{\Gamma}, (\p_\nu u_1)|_{\Gamma}, \p_\nu(\Delta u_1)|_{\Gamma})  \|_{H^{\frac{7}{2}, \frac{3}{2},\frac{5}{2},\frac{1}{2}}(\Gamma)}\\
&\inf_{f\in C_{q_1}^{\Gamma}(k)}\| (u_2|_{\Gamma}-f_1, (\Delta u_2)|_{\Gamma}-f_2, (\p_\nu u_2)|_{\Gamma}-f_3, \p_\nu(\Delta u_2)|_{\Gamma}-f_4)  \|_{H^{\frac{7}{2}, \frac{3}{2},\frac{5}{2},\frac{1}{2}}(\Gamma)}\\
&\le  4 \| (u_1|_{\Gamma}, (\Delta u_1)|_{\Gamma}, (\p_\nu u_1)|_{\Gamma}, \p_\nu(\Delta u_1)|_{\Gamma})  \|_{H^{\frac{7}{2}, \frac{3}{2},\frac{5}{2},\frac{1}{2}}(\Gamma)}\\
 &\| (u_2|_{\Gamma}, (\Delta u_2)|_{\Gamma}, (\p_\nu u_2)|_{\Gamma}, \p_\nu(\Delta u_2)|_{\Gamma})  \|_{H^{\frac{7}{2}, \frac{3}{2},\frac{5}{2},\frac{1}{2}}(\Gamma)} \text{dist}(C_{q_1}^\Gamma(k), C_{q_2}^\Gamma (k)).
\end{aligned}
\end{equation}
This completes the proof of Lemma \ref{lem_int_identity}. 
\end{proof}

An application of the trace theorem gives the following  corollary of  Lemma \ref{lem_int_identity}, see \cite[Theorem 5.1.7]{Agranovich_book}. 
\begin{cor}
\label{cor_int_identity}
Let $q_1, q_2\in L^\infty(\Omega)$ and $k\ge 0$. We have 
\begin{align*}
\bigg|\int_{\Omega}(q_2&-q_1)u_1u_2 dx\bigg|
\le C \| u_1\|_{H^4(\Omega)} \|u_2\|_{H^4(\Omega)} \emph{\text{dist}}(C_{q_1}^\Gamma(k), C_{q_2}^\Gamma (k)),
\end{align*}
for any $u_1, u_2\in H^4(\Omega)$ such that 
\begin{align*}
(\Delta^2-k^4+q_1)u_1=0, \quad \text{in}\quad \Omega, \quad u_1|_{\Gamma_0}=(\Delta u_1)|_{\Gamma_0}=0,\\
(\Delta^2-k^4+q_2)u_2=0, \quad \text{in}\quad \Omega, \quad u_2|_{\Gamma_0}=(\Delta u_2)|_{\Gamma_0}=0.
\end{align*}
\end{cor}

\subsection{Completing  of the proof of Theorem \ref{thm_main}}

Let $0\ne \xi\in \R^n$, $n\ge 3$. We write $\xi=(\xi',\xi_n)$, where $\xi'=(\xi_1,\dots, \xi_{n-1})\in \R^{n-1}$.  Assume first that $\xi' \ne 0$. We then define $e(1)=(\frac{\xi'}{|\xi'|},0)$, $e(n)=(0,\dots, 0, 1)$, and let us complete $e(1)$, $e(n)$ to an orthonormal basis in $\R^n$, which we shall denote by $\{e(1), e(2), \dots, e(n)\}$. In this basis, the vector $\xi$ has the coordinate representation $\tilde \xi=(|\xi'|,0,\dots,0,\xi_n )$. If $\xi'=0$, we let $e(1), \dots, e(n)$ be the standard basis in $\R^n$.

For  $\eta^{(1)},\eta^{(2)}\in \R^n$, we denote by $\tilde \eta^{(1)}, \tilde \eta^{(2)}$ their coordinate representations in the  basis $\{e(1), \dots, e(n)\}$, and we have 
\[
\eta^{(1)}\cdot \eta^{(2)}=\tilde \eta^{(1)}\cdot \tilde \eta^{(2)}, \quad \eta^{(1)}_n=\tilde \eta^{(1)}_n,\quad \eta^{(2)}_n=\tilde \eta^{(2)}_n.
\]
Let us denote by $\mu^{(1)}, \mu^{(2)}$ the vectors in $\R^n$ such that 
\begin{equation}
\label{eq_3_0_0}
\tilde \mu^{(1)}=\bigg(-\frac{\xi_n}{|\xi|}, 0,\dots, 0, \frac{|\xi'|}{|\xi|}\bigg) , \quad \tilde \mu^{(2)}=(0,1,0,\dots, 0).
\end{equation}
Thus, we get $|\mu^{(1)}|=|\mu^{(2)}|=1$ and $\mu^{(1)}\cdot\mu^{(2)}=\mu^{(1)}\cdot \xi=\mu^{(2)}\cdot \xi=0$. Let $k\ge 0$ and set
\begin{equation}
\label{eq_3_0_1}
\begin{aligned}
 \zeta_1=-\frac{ \xi}{2}+\sqrt{k^2+a^2-\frac{|\xi|^2}{4}}\mu^{(1)}+i a \mu^{(2)},\\
 \zeta_2=-\frac{\xi}{2}-\sqrt{k^2+a^2-\frac{|\xi|^2}{4}} \mu^{(1)}-i a \mu^{(2)},
\end{aligned}
\end{equation}
where $a\in \R$ is such that $k^2+a^2\ge \frac{|\xi|^2}{4}$. Note that $\zeta_j\cdot\zeta_j=k^2$, $j=1,2$. 

Let $q_j\in (L^\infty\cap H^s)(\Omega)$ with some $0<s<1/2$, $j=1,2$.  We extend $q_j$ by zero to $\R^n\setminus \Omega$ and denote these extensions by the same letters. It follows from  \cite[Theorem 5.1.10]{Agranovich_book} that $q_j\in (L^\infty\cap H^s)(\R^n)$. 

We shall next construct complex geometric optics solutions to the equations
\begin{equation}
\label{eq_3_0}
(\Delta^2-k^4+q_j) u_j=0\quad \text{in}\quad \Omega,
\end{equation}
which satisfy the following conditions,
\begin{equation}
\label{eq_3_1}
u_j|_{\Gamma_0}=(\Delta u_j)|_{\Gamma_0}=0,
\end{equation}
$j=1,2$. 
Following \cite{Isakov_local}, in order to fulfill the condition \eqref{eq_3_1}, we reflect $\Omega$ with respect to the plane $x_n=0$ and denote this reflection by 
\[
\Omega^*:=\{(x',-x_n)\in \R^n :x=(x',x_n)\in \Omega\},
\]
where $x'=(x_1,\dots, x_{n-1})$.  Let us denote by $q_j^{\text{even}}$ the even extension of $q_j|_{\R^n_-}$ to $\R^n_+$, 
\begin{equation}
\label{eq_3_1_1}
q_j^{\text{even}}(x)=\begin{cases} q_j(x',x_n), & x_n<0,\\
q_j(x', -x_n), & x_n>0.
\end{cases}
\end{equation}
By \cite[Theorem 3.5.1]{Agranovich_book},  we see that  $q_j^{\text{even}}\in (L^\infty\cap H^s)(\R^n)$. 

Let $B=B(0,R) $ be a ball in $\R^n$, centered at $0$, of radius $R\ge 1$, such that $\Omega\cup \Omega^*\subset\subset B$. By Proposition \ref{prop_cgo_Hahner}, there are constants $C_0>0$ and $C_1>0$, depending on $B$ and $n$ only, such that  for  $|\text{Im}\zeta_j|=a\ge \max \{C_0\sqrt{M}, 1\}$, the equation
\begin{equation}
\label{eq_3_2}
(\Delta^2-k^4+q_j^{\text{even}}) \tilde u_j=0\quad \text{in}\quad B,
\end{equation}
has a solution of the form
\begin{equation}
\label{eq_3_3}
\tilde u_j(x)=e^{i\zeta_j\cdot x}(1+r_j(x)),
\end{equation}
where $r_j\in L^2(B)$ satisfies 
\begin{equation}
\label{eq_3_4}
\|r\|_{L^2(B)}\le \frac{C_1}{a^2}\|q_j\|_{L^\infty(\Omega)}.
\end{equation}
By the interior elliptic regularity,  we have  $\tilde u_j\in H^4(\Omega\cup \Omega^*)$, and in view of \eqref{eq_3_2}, we have for all $k\ge 1$, 
\begin{equation}
\label{eq_3_5}
\|\tilde u_j\|_{H^4(\Omega\cup \Omega^*)}\le C k^4 \|\tilde u_j\|_{L^2(B)},
\end{equation}
see \cite[Theorem 6.29]{Grubb_book}.
It follows from \eqref{eq_3_5} that 
\begin{equation}
\label{eq_3_6}
\|\tilde u_j\|_{H^4(\Omega\cup \Omega^*)}\le C k^4 e^{aR}.
\end{equation}

Now let 
\begin{equation}
\label{eq_3_7}
u_j(x)=\tilde u_j(x',x_n)-\tilde u_j(x',-x_n), \quad x\in \Omega. 
\end{equation}
We have $u_j\in H^4(\Omega)$ and $u_j$ satisfies \eqref{eq_3_0} and \eqref{eq_3_1}.

By Corollary \ref{cor_int_identity} and \eqref{eq_3_6}, we get  for all $k\ge 1$ and $a\ge \max \{C_0\sqrt{M}, 1\}$ satisfying $k^2+a^2\ge \frac{|\xi|^2}{4}$, 
\begin{equation}
\label{eq_3_8}
\begin{aligned}
\bigg|\int_{\Omega}(q_2&-q_1)u_1u_2 dx\bigg|
\le C e^{2aR}k^8 \emph{\text{dist}}(C_{q_1}^\Gamma(k), C_{q_2}^\Gamma (k)),
\end{aligned}
\end{equation}
where $u_1, u_2$ are given by \eqref{eq_3_7}.  We shall next substitute $u_1, u_2$ given by \eqref{eq_3_7} into \eqref{eq_3_8}. To that end, using \eqref{eq_3_0_1} and \eqref{eq_3_0_0}, we see that 
\begin{equation}
\label{eq_3_9}
\begin{aligned}
e^{i(\zeta_1+\zeta_2)\cdot x}=e^{-i\xi\cdot x}, \quad e^{i(\zeta_1+\zeta_2)\cdot (x',-x_n)}=e^{-i\xi \cdot(x', -x_n)},\\
e^{i (\zeta_1\cdot (x', x_n)+\zeta_2\cdot (x',-x_n))}= e^{-i\xi_-\cdot x}, \quad e^{i (\zeta_1\cdot (x', -x_n)+\zeta_2\cdot (x',x_n))}= e^{-i\xi_+\cdot x},
\end{aligned}
\end{equation}
where 
\begin{equation}
\label{eq_3_10}
\xi_{\pm}=\bigg(\xi', \pm 2\sqrt{k^2+a^2-\frac{|\xi|^2}{4}}\frac{|\xi'|}{|\xi|}\bigg)\in \R^n.
\end{equation}
Substituting $u_1, u_2$ given by \eqref{eq_3_7} into \eqref{eq_3_8}, and using \eqref{eq_3_9} and \eqref{eq_3_4}, we obtain that 
\begin{equation}
\label{eq_3_11}
\begin{aligned}
\bigg|\int_{\Omega}(q_2-q_1)[e^{-i\xi\cdot x} + e^{-i\xi \cdot(x', -x_n)}-e^{-i\xi_-\cdot x}- e^{-i\xi_+\cdot x} ]   dx\bigg|\\
\le C e^{2aR}k^8 \emph{\text{dist}}(C_{q_1}^\Gamma(k), C_{q_2}^\Gamma (k)) + \frac{C}{a^2}
\end{aligned}
\end{equation}
for all $k\ge 1$ and $a\ge \max \{C_0\sqrt{M}, 1\}$ satisfying $k^2+a^2\ge \frac{|\xi|^2}{4}$. Recalling the definition \eqref{eq_3_1_1} of $q_j^{\text{even}}$, and making a change of variable, we get from \eqref{eq_3_11} that 
\begin{align*}
&\bigg|\int_{\Omega\cup \Omega^*}(q^{\text{even}}_1-q^{\text{even}}_2)e^{-i\xi\cdot x}  dx\bigg|\\
&\le 
\bigg|\int_{\Omega} (q_1-q_2) (e^{-i\xi_-\cdot x}+ e^{-i\xi_+\cdot x})dx \bigg|+
C e^{2aR}k^8 \emph{\text{dist}}(C_{q_1}^\Gamma(k), C_{q_2}^\Gamma (k)) + \frac{C}{a^2},
\end{align*}
and therefore, 
\begin{equation}
\label{eq_3_12}
\begin{aligned}
\big|\mathcal{F}(q^{\text{even}}_1-q^{\text{even}}_2)(\xi)\big| \le& \big|\mathcal{F}(q_1-q_2)(\xi_+)\big| + \big|\mathcal{F}(q_1-q_2)(\xi_-)\big|\\
&+
C e^{2aR}k^8 \emph{\text{dist}}(C_{q_1}^\Gamma(k), C_{q_2}^\Gamma (k)) + \frac{C}{a^2},
\end{aligned}
\end{equation}
for all $k\ge 1$ and $a\ge \max \{C_0\sqrt{M}, 1\}$ satisfying $k^2+a^2\ge \frac{|\xi|^2}{4}$. 

In view of \eqref{eq_3_10}, we have
\[
|\xi_\pm |=\frac{|\xi'|}{|\xi|}2\sqrt{k^2+a^2},
\]
and therefore, by Proposition \ref{prop_Riemann_Lebesgue}, we get for all $N\in \N$ and $0<\tau<1$, 
\begin{equation}
\label{eq_3_14}
|\mathcal{F}(q_1-q_2)(\xi_\pm)|\le \frac{C_N}{\bigg(1+ \tau \frac{|\xi'|}{|\xi|}2\sqrt{k^2+a^2}\bigg)^N}+ C\tau^s.
\end{equation}
 With  $1\le \rho\le \sqrt{k^2+a^2}$ to be chosen, let us consider the set 
\[
E(\rho)=\{\xi\in\R^n: |\xi'|\le \rho, |\xi_n|\le \rho\}.
\]
An application of  Parseval's formula gives 
\begin{equation}
\label{eq_3_19}
\begin{aligned}
\|q^{\text{even}}_1-q^{\text{even}}_2 \|^2_{H^{-1}(\R^n)}\le \bigg( \int_{E(\rho)}+\int_{ \R^n\setminus E(\rho)} \bigg)\frac{|\mathcal{F}(q^{\text{even}}_1-q^{\text{even}}_2)(\xi)|^2}{1+|\xi|^2}d\xi\\
\le   \int_{ E(\rho)} \frac{|\mathcal{F}(q^{\text{even}}_1-q^{\text{even}}_2)(\xi)|^2}{1+|\xi|^2}d\xi   +C\frac{1}{\rho^2}.
\end{aligned}
\end{equation}
We shall now estimate the integral in the right hand side of \eqref{eq_3_19}. 
To this end, using \eqref{eq_3_12}, we get  
\begin{equation}
\label{eq_3_20}
\begin{aligned}
\int_{E(\rho)} \frac{|\mathcal{F}(q^{\text{even}}_1-q^{\text{even}}_2)(\xi)|^2}{1+|\xi|^2}d\xi \le C e^{4aR}k^{16} \rho^n\text{dist}(C_{q_1}^\Gamma(k), C_{q_2}^\Gamma (k))^2 + \frac{C}{a^4}\rho^n\\
+ 4\int_{E(\rho)}  \big|\mathcal{F}(q_1-q_2)(\xi_+)\big|^2d\xi+ 4\int_{E(\rho)}  \big|\mathcal{F}(q_1-q_2)(\xi_-)\big|^2d\xi.
\end{aligned}
\end{equation}
Here we have used the inequality $(a+b+c+d)^2\le 4( a^2+b^2+c^2+d^2)$, $a,b,c,d\in \R$.

In view of \eqref{eq_3_14}, we have for all $N\in \N$ and $0<\tau<1$, 
\begin{equation}
\label{eq_3_21}
\begin{aligned}
\int_{E(\rho)}  \big|\mathcal{F}(q_1-q_2)(\xi_\pm)\big|^2d\xi\le C\rho^n\tau^{2s}+\int_{E(\rho)}  \frac{C_N}{\bigg(1+ \tau \frac{|\xi'|}{|\xi|}2\sqrt{k^2+a^2}\bigg)^N}d\xi.
\end{aligned}
\end{equation}
Using that $|\xi|\le 2\rho$ when $\xi\in E(\rho)$, and integrating in  $\xi_n$,   we get 
\begin{equation}
\label{eq_3_22}
\begin{aligned}
\int_{E(\rho)}  \frac{C_N}{\bigg(1+ \tau \frac{|\xi'|}{|\xi|}2\sqrt{k^2+a^2}\bigg)^N}d\xi\le 2\rho \int_{|\xi'|\le \rho}  \frac{C_N}{\bigg(1+ \tau \frac{|\xi'|}{\rho}\sqrt{k^2+a^2}\bigg)^N}d\xi'\\
\le C_N \bigg(\tau \frac{\sqrt{k^2+a^2}}{\rho}\bigg)^{1-n}  \rho\int_0^\infty \frac{y^{n-2}}{(1+y)^N}dy=C\rho^n\frac{1}{(\tau \sqrt{k^2+a^2})^{n-1}}.
\end{aligned}
\end{equation}
Here we have switched to the polar coordinates and chosen $N$ sufficiently large but fixed. Combining \eqref{eq_3_19}, \eqref{eq_3_20}, \eqref{eq_3_21}, and \eqref{eq_3_22}, we get 
\begin{equation}
\label{eq_3_23}
\begin{aligned}
\|q^{\text{even}}_1-q^{\text{even}}_2 \|^2_{H^{-1}(\R^n)}\le C\frac{1}{\rho^2}+ C e^{4aR}k^{16} \rho^n\text{dist}(C_{q_1}^\Gamma(k), C_{q_2}^\Gamma (k))^2 + \frac{C}{a^4}\rho^n\\
+  C\rho^n\tau^{2s}+ C\rho^n\frac{1}{(\tau \sqrt{k^2+a^2})^{n-1}},
\end{aligned}
\end{equation}
for $1\le \rho\le \sqrt{k^2+a^2}$, $0<\tau<1$, $k\ge 1$ and $a\ge \max \{C_0\sqrt{M}, 1\}$.

We shall now choose the small parameter $\tau$ suitably dependent on $k$ and $a$ so that the last two terms in the right hand side of \eqref{eq_3_23} are of the same order of magnitude. To this end, let us take $\tau$ such that 
\[
\tau^2=\frac{1}{(k^2+a^2)^{\frac{n-1}{2s+n-1}}}. 
\]
Thus, \eqref{eq_3_23} gives  that 
\begin{equation}
\label{eq_3_24}
\begin{aligned}
\|q^{\text{even}}_1-q^{\text{even}}_2 \|^2_{H^{-1}(\R^n)}\le C\frac{1}{\rho^2}+ C e^{4aR}k^{16} \rho^n\text{dist}(C_{q_1}^\Gamma(k), C_{q_2}^\Gamma (k))^2 + \frac{C}{a^4}\rho^n\\
+  C\rho^n\frac{1}{(k^2+a^2)^{\frac{(n-1)s}{2s+n-1}}}
\end{aligned}
\end{equation}
for $1\le \rho\le \sqrt{k^2+a^2}$, $k\ge 1$ and $a\ge \max \{C_0\sqrt{M}, 1\}$.  Later we shall choose $a\ge k$, and therefore, \eqref{eq_3_24} implies that 
\begin{equation}
\label{eq_3_25}
\begin{aligned}
\|q^{\text{even}}_1-q^{\text{even}}_2 \|^2_{H^{-1}(\R^n)}\le C\frac{1}{\rho^2}+ C e^{4aR}k^{16} \rho^n\text{dist}(C_{q_1}^\Gamma(k), C_{q_2}^\Gamma (k))^2 \\
+  C\rho^n\frac{1}{(k^2+a^2)^{\frac{(n-1)s}{2s+n-1}}}
\end{aligned}
\end{equation}
for $1\le \rho\le \sqrt{k^2+a^2}$, $k\ge 1$ and $a\ge \max \{C_0\sqrt{M}, 1,k\}$. 

Here we have used that 
\[
\frac{1}{a^4}\le \frac{C}{(k^2+a^2)^{\frac{(n-1)s}{2s+n-1}}},
\]
which follows from the fact that $k\le a$ and $\frac{(n-1)s}{2s+n-1}<2$, as $s\in(0,\frac{1}{2})$.

Choosing 
\[
\rho=(k^2+a^2)^{\alpha}, \quad 0<\alpha=\frac{(n-1)s}{(2s+n-1)(n+2)} <\frac{1}{2},  
\]
we achieve the equality of the first and the third terms on the right hand side of \eqref{eq_3_25}, and \eqref{eq_3_25} gives
\begin{equation}
\label{eq_3_26}
\begin{aligned}
\|q^{\text{even}}_1-q^{\text{even}}_2 \|^2_{H^{-1}(\R^n)}\le \frac{C}{(k^2+a^2)^{2\alpha}}+ C e^{4aR}k^{16} (k^2+a^2)^{\alpha n}\text{dist}(C_{q_1}^\Gamma(k), C_{q_2}^\Gamma (k))^2\\
\le \frac{C}{a^{4\alpha}}+ C e^{4aR}a^{16+2\alpha n} \text{dist}(C_{q_1}^\Gamma(k), C_{q_2}^\Gamma (k))^2\le 
 \frac{C}{a^{4\alpha}}+ C e^{5aR}\text{dist}(C_{q_1}^\Gamma(k), C_{q_2}^\Gamma (k))^2,
\end{aligned}
\end{equation}
for 
\begin{equation}
\label{eq_3_26_new_100}
a\ge \max \{C_0\sqrt{M}, k\}, \quad k\ge 1.
\end{equation}
Letting 
\[
\delta:=\text{dist}(C_{q_1}^\Gamma(k), C_{q_2}^\Gamma (k)),
\]
and using the fact that $0<\delta<\frac{1}{e}$, we finally choose 
\[
a=C_0\sqrt{M} k+\frac{\log\frac{1}{\delta}}{5R}. 
\]
This choice of $a$ is motivated, on the one hand, by the constraint \eqref{eq_3_26_new_100}, and on the other hand by the requirement that the right hand side of the resulting bound should vanish as $\delta \to 0^+$, to be able to recover the uniqueness result.

It follows therefore  from \eqref{eq_3_26} that 
\begin{equation}
\label{eq_3_27}
\begin{aligned}
\|q^{\text{even}}_1-q^{\text{even}}_2 \|^2_{H^{-1}(\R^n)}\le  
 \frac{C}{(k+\log \frac{1}{\delta})^{4\alpha}}+ e^{Ck}\delta
\end{aligned}
\end{equation}
for all $k\ge 1$. 

Let $0\ne \varphi\in C^\infty_0(\Omega)$. Then using \eqref{eq_3_27}, we get 
\begin{equation}
\label{eq_3_28}
\begin{aligned}
\bigg| \int_{\Omega} (q_1-q_2)\varphi dx\bigg|= \bigg| \int_{\Omega} (q^{\text{even}}_1-q^{\text{even}}_2)\varphi dx\bigg|\le
\|q^{\text{even}}_1-q^{\text{even}}_2 \|_{H^{-1}(\R^n)}\|\varphi\|_{H^1(\Omega)} \\
\le \bigg(  \frac{C}{(k+\log \frac{1}{\delta})^{2\alpha}}+ e^{Ck}\delta^{\frac{1}{2}} \bigg)\|\varphi\|_{H^1(\Omega)}.
\end{aligned}
\end{equation}
The bound \eqref{eq_int_2} follows from \eqref{eq_3_28} by recalling that 
\[
\|v\|_{H^{-1}(\Omega)}=\sup_{0\ne \varphi\in C^\infty_0(\Omega)}\frac{|\langle v,\varphi\rangle_{\Omega}|}{\|\varphi\|_{H^1(\Omega)}}, 
\]
where $\langle \cdot,\cdot\rangle_{\Omega}$ is the distributional duality on $\Omega$. This completes the proof of Theorem \ref{thm_main}.

\subsection{Proof of Corollary \ref{cor_main}}

We follow the classical argument due to Alessandrini \cite{Alessandrini}, see also   \cite{Caro_Marinov_2016}. Let $\varepsilon>0$ be such that $s=\frac{n}{2}+2\varepsilon$. Then by the Sobolev embedding, interpolation and the a priori bounds for $q_j$, we get for all $k\ge 1$, 
\begin{align*}
\|q_1-q_2\|_{L^\infty(\Omega)}\le C\|q_1-q_2 \|_{H^{\frac{n}{2}+\varepsilon}(\Omega)}\le C\|q_1-q_2 \|_{H^{-1}(\Omega)}^{\frac{\varepsilon}{1+s}}\|q_1-q_2 \|_{H^s(\Omega)}^{\frac{1-\varepsilon+s}{s+1}(\Omega)}\\
\le C(2M)^{\frac{1-\varepsilon+s}{s+1}}\|q_1-q_2 \|_{H^{-1}(\R^n)}^{\frac{\varepsilon}{1+s}}\le  \bigg( \frac{C}{(k+\log \frac{1}{\delta})^{2\alpha}}+ e^{Ck}\delta^{\frac{1}{2}} 
\bigg)^{\frac{s-\frac{n}{2}}{2(s+1)}}.
\end{align*}
This completes the proof of Corollary \ref{cor_main}.

\section*{Acknowledgements}
The author would like to thank Katya Krupchyk for her support and guidance, as well as the referee and the handling editor for numerous helpful remarks and suggestions. The research is partially supported by the National Science Foundation (DMS 1815922).

\end{document}